\documentclass[11pt,twoside]{article}

\usepackage[utf8]{inputenc}
\usepackage{hyperref}
\usepackage{amsmath}
\usepackage{amssymb}
\usepackage{times}
\usepackage{anyfontsize}
\usepackage{changepage}
\usepackage{booktabs}
\usepackage{stmaryrd}
\usepackage{url}
\usepackage{longtable}
\usepackage{titlesec}
\usepackage{setspace}
\usepackage{amsthm}
\usepackage{fancyhdr}
\usepackage[letterpaper]{geometry}
\usepackage{lineno}
	
\newcommand{\mc}{\mathcal}
\newcommand{\RR}{\mathbb{R}}

\newcommand{\TAB}{\hspace{0.50cm}}
\newcommand{\IFF}{\Leftrightarrow}

\newcommand{\ii}{\mc{I}}

\newcommand{\mm}{\mc{M}}
\newcommand{\nn}{\mc{N}}

\newcommand{\bb}{\mc{B}}

\newcommand{\om}{\lambda^*}

\newcommand{\bez}{\backslash}
\newcommand{\se}{\subseteq}
\newcommand{\es}{\supseteq}

\newtheorem{defi}{Definition}[section]
\newtheorem{myth}{Theorem}[section]
\newtheorem{prop}{Proposition}[section]
\newtheorem{lem}{Lemma}[section]

\titleformat*{\section}{\filcenter}
\titlelabel{\thetitle. }

\title{Zacny tytuł}
\date{}
\author{Marcin Michalski}

\begin{document}
\thispagestyle{empty}
\begin{flushright}
\begin{spacing}{0}
\textit{\footnotesize key words:}
\\
\textit{\footnotesize measure, category, decomposition, full subset, Luzin, nonmeasurable set}
\end{spacing}

\end{flushright}
\hfill
\begin{flushleft}
Marcin MICHALSKI\footnote{Wrocław University of Science and Technology, Wybrzeże Wyspiańskiego 27, 50-370 Wrocław. The author was supported by grant S50129/K1102 (0401/0086/16), Faculty of Fundamental Problems of Technology.}
\end{flushleft}
\vspace{0.5cm}
\begin{center}
\textbf{{\fontsize{13}{13}\selectfont REDISCOVERED THEOREM OF LUZIN}}
\end{center}
\vspace{0.5cm}
\begin{adjustwidth}{0.5cm}{0pt}
\footnotesize In 1934 N. N. Luzin proved in his short (but dense) paper \textit{Sur la decomposition des ensembles} that every set $X\subseteq \RR$ can be decomposed into two full, with respect to Lebesgue measure or category, subsets. We will try to (at least partially) decipher the reasoning of Luzin and prove this result following his idea.
\end{adjustwidth}
\vspace{0.5cm}
{\normalsize
\section{INTRODUCTION AND PRELIMINARIES}
}
\TAB Although there are more recent and stronger results on topic of decompositions, involving advanced methods (e.g. Gittik-Shelah theorem in \cite{G-S}), there is a unique charm in classic results. In this paper we will provide a detailed proof of one of such results, namely, the following theorem of Luzin (\cite{Luzin1})

\begin{myth}
Every set $X\subseteq \RR$ can be decomposed into two full, with respect to Lebesgue measure or category, subsets.
\end{myth}

A need for such a endeavor arose in us mainly for two reasons. The first is curiosity since the formulation of the theorem gives an impression that it can be proved in a quite simple way. As it has occurred it is not the case. The second reason is that the paper of Luzin is not, let us say, reader -friendly. It has many shortcomings of works published at that time - it lacks precise definitions, most claims come without a proof and key reasonings are hard to interpret due to a relatively frivolous language. In the case of the category Luzin lost us right after defining the set $T$ and sets $H_n$. We took off from there on our own with the Lemma \ref{Lemma for Luzin} and the rest of of our reasoning seems not to correspond to that of Luzin. In the case of measure we managed to follow Luzin quite well, although parts where Lemma \ref{Fubini for Lebesgue outer measure} are needed, as well, as the finish, where one has to invoke the ccc property and Borel hulls, came originally without the proof.
\\
It should be also noted that we are not the only ones that found this result captivating. In the same time and independently from us E. Grzegorek and I. Labuda published a marvelous work \cite{GL} on the topic, including the result of Luzin (calling it "forgotten"), with a wide historical overview. They also took different approach - they were much more patient with the lecture of the work of Luzin and managed to translate the reasoning literally to the letter. We strongly recommend their paper for any reader interested in the historical context of the topic and analysis of similar results from that time.

We will use the standard set theoretic notation, e.g. as in \cite{JECH}. We denote the real line by $\RR$. By Greek alphabet letters  $\alpha, \beta, \gamma, \kappa, … $we denote ordinal numbers with an exception of $\lambda$ ($\lambda^*$) which denotes the Lebesgue (outer) measure on $\RR$ and $\RR^2$ (it will be clear from the context). A cardinality of any set $A$ will be denoted by $|A|$. If $|A|\leq \omega$ then we say that $A$ is countable. In the other case we shall say it is uncountable. A $\sigma$-ideal of sets of Lebesgue measure zero (linear or planar- it will be clear from the context) will be denoted by $\nn$. In the same fashion by $\mm$ we will denote a $\sigma$-ideal of sets of the first category (also: meager). Sets not belonging to $\mm$ are called of the second category. We shall denote a $\sigma$-algebra of Borel sets by $\bb$. We say that a Borel set $B$ is $\ii$-positive with respect to a $\sigma$-ideal $\ii$ (or simply: positive) if $B\notin \ii$. Let us recall some notions regarding $\sigma$-ideals and Borel sets.
\begin{defi}
We say that a $\sigma$-ideal $\ii$
\begin{itemize}
\item has a Borel base if $(\forall A\in\ii) (\exists B\in\bb\cap\ii (A\se B)$;
\item has a Borel hull property if for $(\forall A) (\exists B\in\bb) (A\se B \textnormal{ and } (\forall B'\in\bb) (A\se B'\se B) (B\bez B'\in\ii))$. We call such a set $B$ a Borel hull of $A$, $B=[A]_\ii$.
\end{itemize}
\end{defi}
Both $\mm$ and $\nn$ have Borel bases, since every meager set is contained in some $F_\sigma$ meager set and each null set is contained in a $G_\delta$ set of measure zero. Also both $\mm$ and $\nn$ satisfy countable chain condition (briefly: ccc) property (there is no uncountable family of positive pairwise disjoint Borel sets), so they have the Borel hull property. We give a short proof of this fact.

\begin{prop}
If a $\sigma$  $\mc{I}$ is ccc and has a Borel base, then it has the Borel hull property.
\end{prop}
\begin{proof}
Let $A$ be a set. If $A\in\mc{I}$ we are done. If not, let us take Borel $B_0\es A$. If $B_0\bez A$ contains some $\mc{I}$-positive Borel set $C_1$, then let $B_1=B_0\bez C_1$. Let say that we are at the step $\alpha<\omega_1$. If $\bigcap_{\beta<\alpha}B_{\beta}\bez A$ does not contain any $\mc{I}$-positive Borel set, we are done. Otherwise let $C_\alpha$ be that set and set $B_\alpha=\bigcap_{\beta<\alpha}B_{\beta}\bez C_\alpha$. By ccc a family $\{C_\alpha: \alpha<\omega_1\}$ is countable, which means that the above procedure stabilizes at some countable step $\kappa$ and $\bigcap_{\alpha<\kappa}B_{\alpha}$ is the set.
\end{proof}
Let us recall that a set $M$ is $\ii$-measurable if it belongs to $\sigma(\bb\cup\ii)$- a $\sigma$-algebra of sets generated by Borel sets and $\ii$. Now let us specify a notion of fullness.
\begin{defi}
We say that a set $A\se B$ is full in $B$ with respect to $\ii$, if for each $\ii$-measurable set $M$ we have $M\cap A\in\ii \IFF M\cap B\in\ii$.
\end{defi}
We will call a set $A$ comeager in $B$ if $B\bez A$ is meager. Let us give somewhat more intuitive characterization of full subsets with respect to measure.
\begin{prop}\label{being full with respect to measure}
Let $A\se B$. $(\forall M\in\sigma(\bb\cup\nn))(\om(A\cap M)=\om(B\cap M))\IFF$ $A$ is full in $B$ with respect to $\nn$.
\end{prop}
\begin{proof}
Let us consider the nontrivial implication "$\Leftarrow$". Suppose that there is a measurable set $M$ such that $0\neq\om(A\cap M)=\delta <\om(B\cap M)$. Let $G_1$ and $G_2$ be $G_\delta$ sets covering $A\cap M$ and $B\cap M$ respectively such that $\lambda(G_1)=\om(A\cap M)$ and $\lambda(G_2)=\om(B\cap M)$. Then $\om(G_2\bez G_1\cap A)=0$, but $\om(G_2\bez G_1\cap B)\geq \om(G_2\cap B)-\om(G_1)>0$.
\end{proof}
It is easy to check that for sets $A$ and $B$, $\lambda^*(B)<\infty$, $A$ is full in $B$ $\IFF$ $A$ and $B$ have the same outer measure.
\\
Let us conclude this section with the two following facts which we will find useful later.
\begin{prop}\label{countable union of full subsets of decomposition is full}
If $A=\bigcup_{n\in\omega} A_n$ and we have sets $B_n, n\in\omega$, full in $A_n$, then $\bigcup_{n\in\omega}B_n$ is full in $A$.
\end{prop}
\begin{prop}\label{Continuity from below of outer Lebesgue measure}
Lebesgue outer measure $\lambda^*$ is continuous from below.
\end{prop}
\begin{proof}
Let $(A_n: n\in\omega)$ be a sequence of sets, $\bigcup_{n\in\omega}A_n=A$. Let $\widetilde{A}$ be a Borel hull of $A$ and $\widetilde{A}_n$ be a Borel hull of $A_n$, $n\in\omega$. Then $\lambda^*(\bigcup_{k<n}A_k)=\lambda(\bigcup_{k<n}\widetilde{A}_k)\rightarrow^{n\rightarrow\infty}\lambda(\bigcup_{n\in\omega}\widetilde{A}_n)=\lambda(\widetilde{A})=\lambda^*(A)$ by continuity of $\lambda$.
\end{proof}

{\normalsize\section{DECOMPOSITIONS}}
\begin{center}
\footnotesize{2.1 THE CASE OF THE CATEGORY}
\end{center}
Let $X\se\RR$ be of the second category (otherwise we have nothing to do), enumerated $(X=\{x_\alpha: \alpha<\kappa\})$, and let $\preceq$ be an order on $X$ given by the enumeration. Thanks to Proposition \ref{countable union of full subsets of decomposition is full} we may assume that $X\se [0, 1]$.
\\
Without loss of generality we may assume that every initial segment of $X$ is meager. Otherwise we take the smallest initial segment $I_0$ of $X$ which is of the second category and remove its hull $[I_0]$ from $X$. We repeat the procedure for $X\bez [I_0]$, find $I_1$ - the smallest initial segment of $X\bez [I_0]$ and remove its hull from $X\bez [I_0]$ and so on. By ccc we will stop after countably many steps. Now if we have a collection $\{I_n: n\in\omega\}$ of these smallest initial segments of the second category, then it is sufficient to find decompositions of sets $I_n$, $n\in\omega$. The rest follows by Proposition \ref{countable union of full subsets of decomposition is full}.
\\
Let us fix a set $T\se\RR^2$ as follows
\begin{linenomath*}
\[T=\{(x,y)\in X\times X: y\prec x\},\]
\end{linenomath*}
where $\prec$ is the strict well order determined by the enumeration of $X$. Let us observe that each vertical slice $T_x$ is meager and each horizontal slice $T^y$ is equal to $X$ except some set of the first category. For every $x\in X$ and $n\in \omega$ let us denote nowhere dense sets $F_x^n$ such that $\bigcup_{n\in\omega}F^n_x=T_x$. We may assume that for fixed $x\in X$ sets $F_x^n, n\in\omega,$ are pairwise disjoint. Clearly
\begin{linenomath*}
\[
X=\bigcup_{x\in X}\bigcup_{n\in\omega}(\{x\}\times F^n_x)=\bigcup_{n\in\omega}\bigcup_{x\in X}(\{x\}\times F^n_x).
\]
\end{linenomath*}
For every $n\in\omega$ let us denote $\bigcup_{x\in X}(\{x\}\times F^n_x)$ by $H_n$. Let $\{I_n: n\in\omega\}$ be an enumeration of intervals with rational endpoints which have intersection with $X$ of the second category. Before we continue we will prove the following lemma.
\begin{lem}\label{Lemma for Luzin}
Let a set $X\se\RR$ be nonmeager and $A\se\RR^2$ such that for every $x\in\RR$ $A_x$ is nowhere dense. Then a set
\begin{linenomath*}
\[
Y=\{y: X\bez A^y\in\mm\}
\]
\end{linenomath*}
is nowhere dense.
\end{lem}
\begin{proof}
Suppose that it is not. Then there exists a countable set $Q\se Y$ dense in $Y$. Then $\bigcup_{y\in Q}X\bez A^y\in\mm$, so
\begin{linenomath*}
\[
X\bez (\bigcup_{y\in Q}X\bez A^y)=\bigcap_{y\in Q}A^y
\]
\end{linenomath*}
is nonmeager and thus nonempty. It means that there is $x\in \bigcap_{y\in Q}A^y$  and
\begin{linenomath*}
\[
(\forall y\in Q)(x\in A^y)\equiv (\forall y\in Q)(y\in A_x)
\]
\end{linenomath*}
which means $Q\se A_x$, a contradiction with $A_x$ being nowhere dense.
\end{proof}
Now, let us consider a set
\begin{linenomath*}
\[
Y=\{y\in X: (\forall k\in\omega) (\exists^{\infty}n)(H^y_n\cap I_k\notin\mm)\}.
\]
\end{linenomath*}
We claim that $Y$ is comeager in $X$. Suppose that $B=X\bez Y$ is nonmeager.
\begin{linenomath*}
\[
B=\{y\in X: (\exists k\in\omega) (\forall^{\infty}n)(H^y_n\cap I_k\in\mm)\}
\]
\end{linenomath*}
Then there exist $n_0, k_0\in\omega$ such that a set
\begin{linenomath*}
\[
B_{n_0,k_0}=\{y:\in X: (\forall n>n_0)(H_n^y\cap I_{k_0}\in\mm)\}
\]
\end{linenomath*}
is nonmeager. Let us denote $H=\bigcup_{n\leq n_0}H_n$. Clearly, for every $x\in X$ the slice $H_x$ is nowhere dense. For every $y\in X$ we have that $\bigcup_{n\in\omega}H_n^y$ is comeager in $X$ and, since for each $y\in B_{n_0,k_0}$ a set $\bigcup_{n> n_0}H_n^y\cap I_{k_0}$ is meager, $H^y$ is comeager in $X\cap I_{k_0}$ for nonmeager many $y$. On the other hand by Lemma \ref{Lemma for Luzin} we have that $\{y: X\cap I_{k_0}\bez H^y\in\mm\}$ is nowhere dense which brings a contradiction.
\\
Therefore $Y$ is nonempty so let us pick $y\in Y$. Let us construct by induction two sequences of sets $(X^1_k: k\in\omega)$ and $(X^2_k: k\in\omega)$. For $k=0$ we have infinitely many $n's$ such that $H^y_n\cap I_0\notin\mm$, so let us pick two, say, $n_0^1$ and $n_0^2$ and let us set $X_0^1=H^y_{n_0^1}$ and $X_0^2=H^y_{n_0^2}$. Notice that they are disjoint, since sets $F^n_x$ were pairwise disjoint. Let us assume that at the step $k\in\omega$ we have already sequences $(X^1_m: m<k)$ and $(X^2_m: m<k)$, for which $X^1_m=H^y_{n_m^1}$ and $X^2_m=H^y_{n_m^2}$ for all $m<k$. We still have infinitely many natural numbers $n$ distinct from each of $n_m^i$, $i\in\{1, 2\}, m<k,$ such that $H^y_n\cap I_k\notin\mm$, so let us pick $n^1_k$ and $n^2_k$ such that $H^y_{n^i_k}\cap I_k\notin\mm, i\in\{1, 2\}$ and set $X^i_k=H^y_{n^i_k}, i\in\{1, 2\}$. Sets $X_1=\bigcup_{n\in\omega}X^1_n$ and $X_2=\bigcup_{n\in\omega}X^2_n$ constitute the desired decomposition.
\begin{center}
\footnotesize{2.2 THE CASE OF MEASURE}
\end{center}

Similarly to the previous case let $X\se[0,1]$ be a set of positive outer measure,  $X=\{x_\alpha: \alpha<\kappa\}$, such that its initial segments are null. Again, let us fix a set $T\se\RR^2$ as follows
\begin{linenomath*}
\[
T=\{(x,y)\in X\times X: y\prec x\},
\]
\end{linenomath*}
where $\prec$ is the strict well order determined by the enumeration of $X$. Let us observe that each vertical slice $T_x$ is null and each horizontal slice $T^y$ is equal to $X$ except some set of measure zero. Let $\epsilon>0$. For every $x\in X$ and $n\in \omega$ let us denote intervals with rational endpoints $I_x^n$ such that $\bigcup_{n\in\omega}I^n_x\es T_x$ and $\lambda(\bigcup_{n\in\omega}I_x^n)<\epsilon$. Clearly
\begin{linenomath*}
\[
X=\bigcup_{x\in X}\bigcup_{n\in\omega}(\{x\}\times (I^n_x\cap T_x))=\bigcup_{n\in\omega}\bigcup_{x\in X}(\{x\}\times (I^n_x\cap T_x)).
\]
\end{linenomath*}
For every $n\in\omega$ let us denote $\bigcup_{x\in X}(\{x\}\times I^n_x)$ by $H_n$. For every $y\in X$ we have that $\bigcup_{n\in\omega}H_n^y=X$ (mod $\nn$), therefore by Proposition \ref{Continuity from below of outer Lebesgue measure} for every $y\in X$ there exists $n_y\in\omega$ such that $\lambda^*(\bigcup_{k<n_y}H_k^y)>\lambda^*(X)-\epsilon$. Let us denote
\begin{linenomath*}
\[
Y_n=\{y\in X: \lambda^*(\bigcup_{k<n}H_k^y)>\lambda^*(X)-\epsilon\}.
\]
\end{linenomath*}
Clearly $\bigcup_{n\in\omega}Y_n=X$, so again by Proposition \ref{Continuity from below of outer Lebesgue measure} there exists $N\in\omega$ such that $\lambda^*(\bigcup_{k<N}Y_k)>\lambda^*(X)-\epsilon$. For $y\in\bigcup_{k<N}Y_k$ we have $\lambda^*(\bigcup_{k<N}H_k^y)>\lambda^*(X)-\epsilon$. Let us denote $\bigcup_{k<N}H_k$ by $H$. Before we continue we will prove the following lemma.

\begin{lem}\label{Fubini for Lebesgue outer measure}
Let $X\se[0,1]^2$. Assume that there is a set $A\se [0,1]$ such that $\lambda^*(A)=a$ and $(\forall x\in A)( \lambda^*(X_x)\geq b)$. Then $\lambda^*(X)\geq ab$.
\end{lem}
\begin{proof}
Let $G\es X$ be such that $\lambda(G)=\lambda^*(X)$. Without loss of generality let us assume that $G\se \widetilde{A}\times [0,1]$, $\widetilde{A}$ - the Borel hull of $A$ with respect to measure. Let us consider a function $f: [0,1]\rightarrow [0,1]$, $f(x)=\lambda(G_x)$. Since $f$ is measurable we may write
\begin{linenomath*}
\begin{align*}
\lambda(G)&=\int_{[0,1]}f(x)d \lambda(x)=\int_{\widetilde{A}}f(x)d \lambda(x)=\\
&=\int_{\{x\in\widetilde{A}: f(x)\geq b\}}f(x)d \lambda(x)+\int_{\{x\in\widetilde{A}: f(x)< b\}}f(x)d \lambda(x).
\end{align*}
\end{linenomath*}
Since $\{x\in\widetilde{A}: f(x)\geq b\}\es A$ we have that $\lambda(\{x\in\widetilde{A}: f(x)\geq b\})=a$ and $\lambda(\{x\in\widetilde{A}: f(x)< b\})=0$, so
\[
\lambda^*(X)=\lambda(G)=\int_{\{x\in\widetilde{A}: f(x)\geq b\}}f(x)d \lambda(x)\geq b\int_{\{x\in\widetilde{A}: f(x)\geq b\}}1 \,d \lambda(x)=ab.
\]
\end{proof}
Now, by Lemma \ref{Fubini for Lebesgue outer measure} we have that $\lambda^*(H)>(\lambda^*(X)-\epsilon)^2$. Let $\{G_n: n\in\omega\}$ be a family of finite unions of intervals with rational endpoints such that $\{G_n: n\in\omega\}=\{\bigcup_{k<N}I_x^k: x\in X\}$. For every $n\in\omega$ let $A_n=\{x\in X: G_n=\bigcup_{k<N}I_x^k \}$. We see that we cannot separate sets $A_n$ with their Borel hulls $[A_n]$, otherwise we would have $H\se \bigcup_{n\in\omega}([A_n]\times G_n)$ which has a measure $<\om(X)\cdot\epsilon$. So there are $n_0\neq m_0$ such that $[A_{n_0}]\cap[A_{m_0}]$ is of positive measure. Sets $X^1_0=A_{n_0}\cap[A_{n_0}]\cap[A_{m_0}]$ and $X^2_0=A_{m_0}\cap[A_{n_0}]\cap[A_{m_0}]$ constitute a desired decomposition of $X\cap [A_{n_0}]\cap[A_{m_0}] $. We repeat the whole procedure for $X_0=X\bez [A_{n_0}]\cap[A_{m_0}]$ and so forth; by ccc we shall have a decomposition of $X$ after countably many steps.

\end{document}